\documentclass[11pt,reqno]{amsart}

\usepackage{amsfonts}
\usepackage{amsmath,amscd}
\usepackage{amssymb}
\usepackage{tikz-cd}
\usepackage{amsthm}
\usepackage{cases}
\usepackage{url}
\usetikzlibrary{decorations.markings}

\usepackage{galois}
\usepackage{chemarrow}
\usepackage[all]{xy}
\usepackage{graphicx}
\usepackage[colorlinks,
            linkcolor=blue,
            anchorcolor=black,
            citecolor=red
            ]{hyperref}	
\usepackage{mathrsfs}
\usepackage{extarrows}
\usepackage{texnames}

\def\Hom{{\rm Hom}}

\theoremstyle{plain}

\newtheorem{theorem}{Theorem}[section]

\newtheorem{proposition/example}[theorem]{Proposition/Example}
\newtheorem{proposition}[theorem]{Proposition}
\newtheorem{corollary}[theorem]{Corollary}
\newtheorem{lemma}[theorem]{Lemma}

\newtheorem{conjecture}[theorem]{Conjecture}

\theoremstyle{definition}
\newtheorem{definition}[theorem]{Definition}
\newtheorem{remark}[theorem]{Remark}
\newtheorem{example}[theorem]{Example}
\newtheorem{fact}[theorem]{Fact}

\newtheorem{conjecture/question}[theorem]{Conjecture/Question}

\newtheorem{remark/definition}[theorem]{Remark/Definition}
\newtheorem{definition/notation}[theorem]{Definition/Notation}

\numberwithin{equation}{section}

\newcommand{\dynkinradius}{.06cm}
\newcommand{\dynkinstep}{.8cm}
\newcommand{\dynkindot}[2]{\fill (\dynkinstep*#1,\dynkinstep*#2) circle (\dynkinradius);}

\newcommand{\dynkinline}[4]{\draw[thin] (\dynkinstep*#1,\dynkinstep*#2) -- (\dynkinstep*#3,\dynkinstep*#4);}
\newcommand{\dynkindots}[4]{\draw[dotted] (\dynkinstep*#1,\dynkinstep*#2) -- (\dynkinstep*#3,\dynkinstep*#4);}

\newenvironment{dynkin}{\begin{tikzpicture}[decoration={markings,mark=at position 0.7 with {\arrow{>}}}]}
{\end{tikzpicture}}

\begin{document}
\title{\textbf{Stability and Indecomposability of the Representations of Quivers of $A_n$-type}}

\author{Pengfei Huang}

\address{ \textsc{School of Mathematics, University of Science and Technology of China, Hefei 230026, China}}

\email{pfhwang@mail.ustc.edu.cn}

\author{Zhi Hu}

\address{ \textsc{Research Institute for Mathematical Sciences, Kyoto University, Kyoto 606-8502, Japan}}

\email{halfask@mail.ustc.edu.cn}


\date{}
\thanks{This paper is completed around May 2017.}

\begin{abstract}
In his paper \cite{MR1}, Markus Reineke proposed a conjecture that there exists a stable weight system $\Theta$ for every indecomposable representation of Dynkin type quiver.  In this paper, we showed  this conjecture is true for  quivers of $A_n$-type  by   combinatorial construction of a special weight system.  We also reinterpret this weight system in terms of semi-invariant theory.
\end{abstract}

\maketitle
\tableofcontents

\section{Introduction}
In his remarkable paper \cite{MR1} published on Invent. Math. in 2003, Markus Reineke proposed the following conjecture:
\begin{conjecture}[Reineke, \cite{MR1}]
If $Q$ is a quiver of Dynkin type, there exists a weight system $\Theta$ on $Q$ such that the stable representations are precisely the indecomposable ones.
\end{conjecture}

If this conjecture is true, it will has some valuable applications. For example, it can be used to study the stratification of representation varieties of Dynkin quiver  \cite{MR1}, and to study identities between products of quantum dilogarithm series associated with Dynkin quivers \cite{BK}\footnote{We thank the referee for pointing out this reference to us.}.
In this paper, we will give an elementary proof of Reineke's conjecture for the quivers of $A_n$-type by  combinatorial construction of a special weight system\footnote{From the communications with Prof. Reineke, Prof. Hille and Prof. Juteau, we know that Hille and Juteau have a proof for this conjecture in $A_n$-case, however it is unpublished. Meanwhile, from the report of the referee, we know that in his paper \cite{BK}, Keller suggests  to resolve this conjecture by the method of \cite{IT}. }.

On the other hand, Juteau \cite{d} has found some counterexamples for Reineke's conjecture in the quivers of $D$- and $E$-type by computer program (unpublished). Therefore, a modified version of original Reineke's conjecture is proposed as follows\footnote{We do not know whether such modified version is true for  the  Dynkin quivers (even for tame quivers), but so far as we know, Juteau and Hille \cite{d,hi} are attempting to prove it.}.
 \begin{conjecture}[= Conjecture \ref{conj}]\label{conj1}
Let $Q$ be a  Dynkin quiver, then the abelian category  $\mathrm{\mathbf{Rep}}_k(Q)$ is a maximal stable category.
\end{conjecture}

  A natural extension is to consider the Reineke-type conjecture for certain triangulated categories with Bridgeland stability conditions\footnote{Very recently, we note that the authors of \cite{KOT} prove that for the derived category $D^b(Q)$ of a Dynkin quiver $Q$, there exists a Bridgeland stability condition $\sigma$ such that for an object $E\in D^b(Q)$ the following are equivalent: i) $E$ is indecomposable, ii) $E$ is exceptional, iii) $E$ is $\sigma$-stable. }.

\section{Preliminaries}
In this section, we collect some basic materials  of  quiver theory (for more details, see \cite{1}). Throughout the paper,  $k$ is assumed to be a fixed algebraically closed field.

\begin{definition}
\begin{enumerate}
  \item A {\em quiver} $Q=(Q_0,Q_1,s,t)$ is a $4$-tuple, where
\begin{itemize}
  \item $ Q_0$ and $Q_1$ are finite sets of {\em vertices} and {\em arrows} respectively,
  \item $ s,t: Q_1\to Q_0$ map each arrow $a\in Q_1$ to its {\em starting vertex} $s(a)$ and {\em terminal vertex} $t(a)$.
\end{itemize}
  \item Let $Q=(Q_0, Q_1)$ be a quiver, $Q'=(Q_0', Q_1')$ is called a {\em subquiver} if
\begin{itemize}
  \item  $Q_0'\subset Q_0$ and $Q_1'\subset Q_1$,
  \item $s(a), t(a)\in Q_0'$ for all $a\in Q_1'$.
\end{itemize}
In particular, a subquiver $Q'$ is called a {\em full subquiver} if furthermore, $a\in Q_1'$ for all $a\in Q_1$ satisfying  $s(a)\in Q_0'$.
  \item  A  {\em path} $\gamma$ of a quiver $Q$ is a sequence $a_1\cdots a_n (n\geq 1)$ of arrows that satisfies $s(a_{i+1})=t(a_i)$ for $1\leq i\leq n-1$, and the starting vertex of $a_1$ and terminal vertex of $a_n$ are called starting vertex  and terminal vertex of $\gamma$, respectively.
\end{enumerate}
\end{definition}

\begin{definition}\begin{enumerate}
                    \item Let $Q=(Q_0, Q_1)$ be a quiver, a {\em representation} of $Q$ over $k$ is a pair $X=\{(X_i)_{i\in Q_0},(X_a)_{a\in Q_1}\}$, where
$(X_i)_{i\in Q_0}$ is a family of finite dimensional  $k$-vector spaces and $(X_a)_{a\in Q_1}$ a family of $k$-linear maps associated to all arrows, i.e.
$X_a: X_{s(a)}\to X_{t(a)}$. Define  $d=(d_i)_{i\in Q_0}\in \mathbb{Z}^{|Q_0|}$ for $d_i=\dim_k X_i$, and call it the {\em dimension vector} of $X$.
Let $X, Y$ be two $k$-representations of $Q$, a {\em morphism} $u: X\to Y$ is a collection of linear maps $u_i: X_i\to Y_i$ for all $i\in Q_0$ such that for each arrow $a\in Q_1$, the following diagram commutes:
$$\begin{CD}
  X_{s(a)} @>X_a>> X_{t(a)} \\
  @V u_{s(a)} VV @V u_{t(a)} VV  \\
  Y_{s(a)} @>Y_a>> Y_{t(a)}.
\end{CD} $$
 We say the morphism $u$ is an {\em isomorphism} if moreover each $u_i$ is an isomorphism, and denote it as $X\cong Y$. We denote by $\mathrm{\mathbf{Rep}}_k(Q)$ the category of representation of $Q$ over $k$.
                  \item The {\em direct sum} $W=X\oplus Y$ of two representations $X$ and $Y$ of $Q$ is defined by the pair
$$W=\{(W_i)_{i\in Q_0},(W_a)_{a\in Q_1}\}=\{(X_i\oplus Y_i)_{i\in Q_0},(X_a\oplus Y_a)_{a\in Q_1}\},$$
where each linear map is given by
$$W_a=X_a\oplus Y_a: X_{s(a)}\oplus Y_{s(a)}\to X_{t(a)}\oplus Y_{t(a)}.$$
A representation $W$ of $Q$ is said to be {\em decomposable} if there exist non-zero representations $X$ and $Y$ such that $W\cong X\oplus Y$, otherwise it is said to be {\em indecomposable}.
\item Let $X$ and $Y$ be two representations of $Q$. $X$ is said to be a {\em subrepresentation} $Y$ if $X_i\subset Y_i$ for all $i\in Q_0$ and $X_a=Y_a|_{X_{s(a)}}: X_{s(a)}\to X_{t(a)}$ for all $a\in Q_1$. A representation is called {\em simple} if it has no proper non-zero subrepresentations, and called {\em semisimple} if it is the direct sum of simple representations.
    \item  We say a representation $X$ of $Q$ is  {\em thin}  if
$\dim_k(X_i)\leq1\quad \text{for all}\,\,\, i\in Q_0$,
that is, if each linear space $X_i$ is either $0$ or $k$.
    \end{enumerate}
\end{definition}

Given a quiver $Q$, an important aim of representation theory is to classify all representations and all morphisms up to isomorphism.  Krull-Schmidt theorem  makes this classification problem easier, it states that every representation of a given quiver can be uniquely decomposed into indecomposable representations up to ordering, so we only need to classify the indecomposable representations. A quiver $Q$ is said to be of finite type if it has finitely many isomorphism classes of indecomposable representations.
Gabriel's classification theorem states that for  a connected quiver $Q$ without oriented cycles the following are equivalent
\begin{itemize}
  \item $Q$ is of finite type,
  \item the underlying graph of $Q$ is  a simply laced Dynkin diagram, namely one of the followings
  $$\begin{dynkin}{$A_n:$}
\foreach \x in {1,...,4}
{
        \dynkindot{\x}{0}
    }
    \dynkinline{1}{0}{2}{0};
    \dynkindots{2}{0}{3}{0};
    \dynkinline{3}{0}{4}{0};
  \end{dynkin}
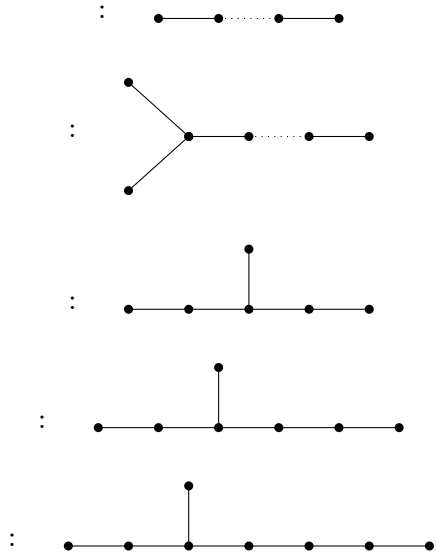
 $$

  $$\begin{dynkin}{$D_n:$}
    \foreach \x in {2,...,5}
    {
        \dynkindot{\x}{0}
    }
    \dynkindot{1}{.9}
    \dynkindot{1}{-.9}
    \dynkindot{2}{0}
    \dynkinline{1}{.9}{2}{0}
    \dynkinline{1}{-.9}{2}{0}
    \dynkinline{2}{0}{3}{0}
    \dynkindots{3}{0}{4}{0}
    \dynkinline{4}{0}{5}{0}
  \end{dynkin}$$

  $$\begin{dynkin}{$E_6:$}
    \foreach \x in {1,...,5}
    {
        \dynkindot{\x}{0}
    }
    \dynkindot{3}{1}
    \dynkinline{1}{0}{2}{0}
    \dynkinline{2}{0}{3}{0}
    \dynkinline{3}{1}{3}{0}
    \dynkinline{3}{0}{4}{0}
    \dynkinline{4}{0}{5}{0}
  \end{dynkin}$$

  $$
  \begin{dynkin}{$E_7:$}
    \foreach \x in {1,...,6}
    {
        \dynkindot{\x}{0}
    }
    \dynkindot{3}{1}
    \dynkinline{1}{0}{2}{0}
    \dynkinline{2}{0}{3}{0}
    \dynkinline{3}{1}{3}{0}
    \dynkinline{3}{0}{4}{0}
    \dynkinline{4}{0}{5}{0}
    \dynkinline{5}{0}{6}{0}
  \end{dynkin}$$

  $$
  \begin{dynkin}{$E_8:$}
    \foreach \x in {1,...,7}
    {
        \dynkindot{\x}{0}
    }
    \dynkindot{3}{1}
    \dynkinline{1}{0}{2}{0}
    \dynkinline{2}{0}{3}{0}
    \dynkinline{3}{1}{3}{0}
    \dynkinline{3}{0}{4}{0}
    \dynkinline{4}{0}{5}{0}
    \dynkinline{5}{0}{6}{0}
    \dynkinline{6}{0}{7}{0}
  \end{dynkin}$$
  \item the quadratic form $$q_Q(\alpha)=\sum\limits_{i\in Q_0}\alpha_i^2-\sum\limits_{a\in Q_1}\alpha_{s(a)}\alpha_{t(a)},$$
  where $\alpha\in \mathbb{Z}^{|Q_0|}$, is positive definite.
\end{itemize}
Moreover one has  the following bijective correspondence:
$$\begin{array}{c}
  \boxed{\mathrm{isomorphism \  classes \ of\  indecomposable\  representations}\  X } \\
  \updownarrow \\
  \boxed{\mathrm{positives\ roots \ of \ the\  quadratic\  form} \ q_Q } \\
  \updownarrow \\
  \boxed{\mathrm{noninitial \ cluster\ variables}\ c_X}.
\end{array}$$

For  a quiver $Q$, we  express the underlying graph $\Gamma_Q$ as a binary set $\Gamma_Q=:\{(Q_0,Q_1')\}$ of vertices and edges, where $Q_1'$ is obtained by taking all arrows in $Q_1$ as edges. For any edge $l\in Q_1'$, it corresponds to a unique arrow $a\in Q_1$, we denote $s(l)=s(a)$ and $t(l)=t(a)$ the starting vertex and terminal vertex of $l$, respectively.

\begin{definition}
The {\em support} of $X$ is a subset of $\Gamma_Q$ consisting of all vertices $i$ with the assigning linear space $X_i\neq0$ and all edges connecting these vertices, that is:
$$\text{supp}(X):=\{(\tilde{Q}_0,\tilde{Q}_1)| i\in \tilde{Q}_0\,\,\text{if}\,\,X_i\neq0; l\in\tilde{Q}_1\,\,\text{if}\,\,s(l),t(l)\in\tilde{Q}_0\}\subset\Gamma_Q.$$
The {\em support quiver} $\text{supp}_X(Q)$ is given by recovering the arrows of all edges in $\text{supp}(X)$, that is, given by the vertices $i$ with $X_i\neq$ and the arrows $a$ with $X_a\neq0$.
\end{definition}
The following facts are very obvious.
\begin{fact} Let $Q$ be a quiver, then
\begin{enumerate}
  \item If $X$ is an indecomposable representation, then its support quiver $\text{supp}_X(Q)$ is connected.
  \item If two representations $X$ and $Y$ are isomorphic, then $\text{supp}_X(Q)=\text{supp}_Y(Q)$.
  \item Let $X$  be a thin representation of $Q$. Then $X$ is indecomposable if and only if the support quiver $\text{supp}_X(Q)$ is connected.
\end{enumerate}

\end{fact}

\begin{definition}[\cite{j}]
\begin{enumerate}
                    \item  Let $\mathcal{A}$ be a category, and  $w,r$ be two functions on $\mathcal{A}$, called {\em weight function} and {\em rank function} respectively, such that $r(X)\neq 0$ for any nonzero object in $\mathcal{A}$. A object  $X\in \mathcal{A}$ is called {\em $(w,r)$-stable} (respectively, {\em $(w,r)$-semistable}) if for any nonzero subobject $U$ of $X$, we have $\mu(U)<\mu(X)$ (respectively, $\mu(U)\leq\mu(X)$), where the slope function $\mu(X)$ is defined by $\mu(X)=\frac{w(X)}{r(X)}$.
                        \item  Let $w,r$ be  weight  function and rank function on a category $\mathcal{A}$ respectively, all $(w,r)$-stable (respectively, $(w,r)$-semistable) objects of $\mathcal{A}$ form a full subcategory of $\mathcal{A}$, called $(w,r)$-stable (respectively, $(w,r)$-semistable) subcategory. Two pairs $(w,r)$ and $(w^\prime,r^\prime)$ is called {\em stable-equivalent} (respectively, {\em semistable-equivalent}) if they induce the same stable (respectively, semistable) subcategories.
                    \item Let $\mathcal{A}$ be  an abelian category, if there exist additive weight function $w$ and rank function $r$ on $\mathcal{A}$ such that $(w,r)$-stable subcategory consists of all indecomposable objects, we call $\mathcal{A}$ be a \emph{maximal   stable category}.
                  \end{enumerate}

\end{definition}

\section{Reineke's Conjecture for Quivers of $A_n$-Type}

\begin{conjecture}[Modified Reineke's conjecture]\label{conj}
Let $Q$ be a  Dynkin quiver, then the abelian category  $\mathrm{\mathbf{Rep}}_k(Q)$ is a maximal stable category.
\end{conjecture}

In this section, we confirm the above conjecture for quivers of type $A_n$. Namely, we will prove the following main theorem.
\begin{theorem}\label{1.3.5}
If $Q$ is a  quiver of type $A_n$, then there exists a weight system  $\Theta=(\theta_i)_{i\in Q_0}\in \mathbb{Z}^{|Q_0|}$ such that the stable representations with respect to the weight function $w(X)=\sum_{i\in Q_0}\theta_i\dim X_i$
and rank function $r(X)=\sum_{i\in Q_0}\dim X_i$ are precisely the indecomposables, namely  $\mathrm{\mathbf{Rep}}_k(Q)$ is a maximal stable category.
\end{theorem}

\subsection{Intrinsic Weight System}
Let $Q$ be a quiver of type $A_n$,
we put it horizontally  and  fix a reference  direction from left to right so that assign numbers $1,\cdots,n$ to the vertices of $Q$ along the reference  direction, then  we can classify all vertices  into the following four types according to the  directions of arrows attached to them:

$\bullet$ a vertex $i\in Q_0$ is called of type I if it is only as the starting vertex of arrows linking to it;

$\bullet$ a vertex $i\in Q_0$ is called of type II if it is only as the terminal vertex of arrows linking to it;

$\bullet$ a vertex $i\in Q_0$ is called of type III if there is  a path with the reference  direction such that $i$ is a vertex but neither a staring nor a terminal one of that path;

$\bullet$ a vertex $i\in Q_0$ is called of type IV if  there is  a path with the  direction opposite to   the reference  direction such that $i$ is a vertex but neither a staring nor a terminal one of that path.

Now we define a weight system  $\Theta=\{\theta_i\}_{i\in Q_0}$ according to the  type of each vertex as follows:
\begin{equation}\label{2.2}
\theta_i=\left\{
             \begin{array}{ll}
              l_i+r_i+2l_ir_i, & \hbox{$i$ \textrm{ is a vertex of type I};} \\
               -l_i-r_i-2l_ir_i, & \hbox{$i$ \textrm{ is a vertex of type II};} \\
r_i-l_i, & \hbox{$i$ \textrm{ is a vertex of type III};} \\
l_i-r_i, & \hbox{$i$ \textrm{ is a vertex of type IV}} \\
             \end{array}
           \right.
\end{equation}
where $l_i$ and $r_i$ stands for the  number of vertices on the  left and right of the vertex $i$, respectively.
Such weight system is called the \emph{intrinsic weight system},
and the corresponding  weight function  $w_\Theta(X)=\sum_{i\in Q_0}\theta_i\dim X_i$ for a representation $X$ of $Q$ is called  the \emph{intrinsic weight function}.

\begin{example}\label{A7}
The following quiver of $A_n$-type
\begin{equation}\label{1}
  \underset{1}{\bullet}\longrightarrow\underset{2}{\bullet}\longrightarrow\underset{3}{\bullet}\longrightarrow
  \underset{4}{\bullet}\longleftarrow\underset{5}{\bullet}\longleftarrow\underset{6}{\bullet}\longrightarrow
  \underset{7}{\bullet},
\end{equation}
has four  types  of vertices:
 \begin{equation*}
   \left\{\begin{split}
     &I: \qquad\quad 1, 6,\\
     &II: \qquad\,\,\, 4, 7,\\
    &III: \qquad 2, 3,\\
     &IV: \qquad\,\, 5,
   \end{split}
   \right.
 \end{equation*}
Then the intrinsic weight system $\Theta$ is given by  $\theta_1=6$, $\theta_2=4$, $\theta_3=2$, $\theta_4=-24$, $\theta_5=2$, $\theta_6=16$, $\theta_7=-6$.
\end{example}

\begin{lemma}\label{total weight}
Let $Q=\{Q_0,Q_1\}$ be a quiver of $A_n$-type with  the intrinsic weight system $\Theta=\{\theta_i\}_{i\in Q_0}$. Then
\begin{enumerate}
  \item along any  path, the weights  at the vertices contained in the path  decrease.
  \item the sum $\sum\limits_{i\in Q_0}\theta_i$ of  weights at all vertices is exactly zero.
\end{enumerate}

\end{lemma}

\begin{proof}
(1) We just need to show along each arrow $a\in Q_1$, the weight decreases, i.e., $\theta_{s(a)}>\theta_{t(a)}$ for each arrow $a\in Q_1$.
We can draw the quiver $Q$ as follows:
$$Q:\quad\cdots\underset{s(a)}{\bullet}\xlongrightarrow{a}\underset{t(a)}{\bullet}\cdots,$$
if we consider the two vertices and arrows closed to $\underset{s(a)}{\bullet}\xlongrightarrow{a}\underset{t(a)}{\bullet}$, then we have the
following four cases:

$(i)$ $Q:\quad\cdots\longrightarrow\underset{s(a)}{\bullet}\xlongrightarrow{a}\underset{t(a)}{\bullet}\longleftarrow\cdots,$ in this case,
we have
$$\theta_{s(a)}=r_{s(a)}-l_{s(a)},\quad \theta_{t(a)}=-l_{t(a)}-r_{t(a)},$$

$(ii)$ $Q:\quad\cdots\longleftarrow\underset{s(a)}{\bullet}\xlongrightarrow{a}\underset{t(a)}{\bullet}\longrightarrow\cdots,$ in this case,
we have
$$\theta_{s(a)}=l_{s(a)}+r_{s(a)}+2l_{s(a)}\cdot r_{s(a)},\quad \theta_{t(a)}=r_{t(a)}-l_{t(a)},$$

$(iii)$ $Q:\quad\cdots\longrightarrow\underset{s(a)}{\bullet}\xlongrightarrow{a}\underset{t(a)}{\bullet}\longrightarrow\cdots,$ in this case,
we have
$$\theta_{s(a)}=r_{s(a)}-l_{s(a)},\quad \theta_{t(a)}=r_{t(a)}-l_{t(a)},$$

$(iv)$ $Q:\quad\cdots\longleftarrow\underset{s(a)}{\bullet}\xlongrightarrow{a}\underset{t(a)}{\bullet}\longleftarrow\cdots,$ in this case,
we have
$$\theta_{s(a)}=l_{s(a)}+r_{s(a)}+2l_{s(a)}\cdot r_{s(a)},\quad \theta_{t(a)}=-l_{t(a)}-r_{t(a)})-2l_{t(a)}\cdot r_{t(a)},$$
Since $r_{s(a)}>r_{t(a)}\geq0$ and $0\leq l_{s(a)}<l_{t(a)}$, in each case, we have $\theta_{s(a)}>\theta_{t(a)}$.\par

(2) For any subquiver $Q'=\{Q_0',Q_1'\}$ of $Q$, we give each vertex $i\in Q_0'$ a new weight:
$$\theta_i^{Q'}=\#\{a|\,\, s(a)=i, \,\,a\in Q_1'\}-\#\{a|\,\, t(a)=i, \,\,a\in Q_1'\}\in \{\pm2, 0\},$$
the difference of numbers of arrows in $Q_1'$ starting at $i$ and numbers of arrows in $Q_1'$ terminating at $i$, this construction immediately gives
$$\sum\limits_{i\in Q_0'}\theta_i^{Q'}=0.$$
Then for each vertex $i\in Q_0$, its weight $\theta_i$ is the sum of all such new weights $\theta_i^{Q'}$ for the connected subquiver $Q'$ contains $i$:
$$\theta_i=\sum\limits_{\substack{Q'\subset Q\,\,\text{connected subquiver}\\ i\in Q_0'}}\theta_i^{Q'}.$$

If $i$ is of type I, then near the vertex $i$, the quiver locally looks like $\longleftarrow\underset{i}{\bullet}\longrightarrow,$ all subquivers contain $i$ can be divided into three classes:

$(i)$ $\cdots\longleftarrow\underset{i}{\bullet}$,

$(ii)$ $\underset{i}{\bullet}\longrightarrow\cdots$,

$(iii)$ $\cdots\longleftarrow\underset{i}{\bullet}\longrightarrow\cdots$.

The first class has $l_i$ subquivers, for each subquiver $Q'$, we have $\theta_i^{Q'}=1$; the second class has $r_i$ subquivers, for each subquiver $Q'$, we have $\theta_i^{Q'}=1$; the third class has $l_ir_i$ subquivers, for each subquiver $Q'$, we have $\theta_i^{Q'}=2$. Consequently,
$$\sum\limits_{\substack{Q'\subset Q\,\,\text{connected subquiver}\\ i\in Q_0'}}\theta_i^{Q'}=l_i+r_i+2l_ir_i=\theta_i.$$

If $i$ is of type III, then near the vertex $i$, the quiver locally looks like $\longrightarrow\underset{i}{\bullet}\longrightarrow,$ all subquivers contain $i$ can be divided into three classes:

$(i)$ $\cdots\longrightarrow\underset{i}{\bullet}$,

$(ii)$ $\underset{i}{\bullet}\longrightarrow\cdots$,

$(iii)$ $\cdots\longrightarrow\underset{i}{\bullet}\longrightarrow\cdots$.

The first class has $l_i$ subquivers, for each subquiver $Q'$, we have $\theta_i^{Q'}=-1$; the second class has $r_i$ subquivers, for each subquiver $Q'$, we have $\theta_i^{Q'}=1$; the third class has $l_ir_i$ subquivers, for each subquiver $Q'$, we have $\theta_i^{Q'}=0$. Hence,
$$\sum\limits_{\substack{Q'\subset Q\,\,\text{connected subquiver}\\ i\in Q_0'}}\theta_i^{Q'}=r_i-l_i=\theta_i.$$
The cases of type II and  IV are similar.

Therefore, the sum of all weights is calculated  as
$$\sum\limits_{i\in Q_0}\theta_i=\sum\limits_{i\in Q_0}\Big(\sum\limits_{\substack{Q'\subset Q\,\,\text{connected subquiver}\\ i\in Q_0'}}\theta_i^{Q'}\Big)=
\sum\limits_{\substack{Q'\subset Q\\ \text{connected subquiver}}}\Big(\sum\limits_{i\in Q_0'}\theta_i^{Q'}\Big)=0.$$
 We complete the proof.
\end{proof}

We denote  the following indecomposable thin  representation (the orientation in the graph is just  an example)
$$\quad 0\longrightarrow \cdots \longrightarrow \underset{p}{k}\xlongrightarrow{1} \cdots \xlongrightarrow{1} \underset{q}{k}\longrightarrow 0\longrightarrow\cdots\longrightarrow0,$$
where $1\leq p\leq q\leq n$, by $I_{p,q}$.
Then the  indecomposable representations of $Q$  are classified by $I_{p,q}$'s, more precisely, we have
\begin{proposition}[\cite{KMS,LM}]
Let $X$ be a representation of a quiver $Q$ of type $A_n$. Then $X$ is indecomposable if and only if $X$ is a thin representation whose support quiver is connected, that is, $X$ is isomorphic to some $I_{p,q}$.
\end{proposition}

\begin{lemma}\label{subweight}Let $Q$ be a quiver of $A_n$-type, and  $X$ is the indecomposable  representation of type $I_{1,n}$, then $X$ is stable with respect to the intrinsic weight function and rank function.
\end{lemma}
\begin{proof}To show the stability, we only need to prove the intrinsic weight function $w_\Theta(X^\prime)$ on any proper subrepresentation $X^\prime$ of $I_{1,n}$ is negative.
Obviously, the support quiver $Q^\prime=\text{supp}_{X'}(Q)$ is a proper full subquiver of $Q$.
We first assume  $Q'$ is connected, then $Q$ must look like as follows:
$$Q:\quad \cdots\bullet\longrightarrow Q' \longleftarrow\bullet \cdots,$$
and   denoting the two vertices of the boundary of $Q'$ by $s(Q')$ and $t(Q')$, one draws $Q$ as follows:
$$Q:\quad \cdots\bullet\longrightarrow\underset{s(Q')}{\bullet}\cdots\underset{t(Q')}{\bullet} \longleftarrow\bullet \cdots.$$
Let $l_{Q'}$ and $r_{Q'}$ denote the number of vertices on the left of the whole $Q'$ and the number of vertices on the right of the whole $Q'$, respectively.

To compute the  weight function $w_\Theta(X^\prime)$, we first separate $Q'$ from $Q$, and view $Q'$ as an independent quiver. Then $Q^\prime$ carries a weight system $\Theta^\prime$, called the independent weight system,  given by the manner described previously so that the sum denoted by $\theta_{\textrm{ind}}(Q')$ of the weights belong to the  independent weight system  is zero. By Lemma \ref{total weight}, the actual weight function $w_\Theta(X^\prime)$ is the sum of $\theta_{\textrm{ind}}(Q')$ and $\theta_{\text{add}}(Q')$, where $\theta_{\text{add}}(Q')$ is the sum of the added new weights at the vertices in $Q^\prime_0$ caused by the connected subquivers containing not only  vertices in $ Q_0'$ but also in $Q_0\backslash Q_0'$.  Such connected quivers are divided into three cases:

 $(i)$ the considered connected subquiver ( inside the box) contains vertices in $Q'$ and some vertices only on the right of $Q'$, like the following:
 $$\cdots\bullet\longrightarrow\underset{s(Q')}{\bullet}\cdots\boxed{\cdots\underset{t(Q')}{\bullet} \longleftarrow\bullet \cdots}\cdots,$$

 $(ii)$ the considered connected subquiver (inside the box) contains vertices in $Q'$ and some vertices only on the left of $Q'$, like the following:
 $$\cdots\boxed{\cdots\bullet\longrightarrow\underset{s(Q')}{\bullet}\cdots}\cdots\underset{t(Q')}{\bullet} \longleftarrow\bullet\cdots,$$

 $(iii)$ the considered connected subquiver ( inside the box) contains the whole $Q'$ and some vertices both on the right and on the left of $Q'$, like the following:
 $$\cdots\boxed{\cdots\bullet\longrightarrow\underset{s(Q')}{\bullet}\cdots\underset{t(Q')}{\bullet} \longleftarrow\bullet\cdots}\cdots.$$

 The first case  includes  $r_{Q'}\cdot|Q_0'|$ choices. A key observation  is that   each choice contributes a term $-1$ to  the sum of  weights at the vertices of $Q'$.
Indded, let $\tilde Q$ be a such connected subquiver which produce new weight for the vertices in $\tilde Q_0$ as  in the proof of Lemma \ref{total weight}
$$\theta_i^{\tilde{Q}}=\#\{a|\,\, s(a)=i, \,\,a\in \tilde{Q}_1\}-\#\{a|\,\, t(a)=i, \,\,a\in \tilde{Q}_1\}.$$
 Then once we compute the sum $\sum\limits_{i\in\tilde{Q}_0\cap Q'_0}\theta_i^{\tilde Q}$, the inner arrows of $\tilde{Q}\cap Q'$ do no work, only the arrow  closest attaching  to $\tilde{Q}\cap Q'$ has effect   by providing one term -1 in the sum.
Similarly, the second case admits  $l_{Q'}\cdot|Q_0'|$ choices, and  each case contributes a term  $-1$ to  the sum; the third case contains $l_{Q'}\cdot r_{Q'}$  choices, and  each case contributes a term  $-2$ to  the sum. Finally we reach
 $$\theta_{\text{add}}(Q')=-r_{Q'}\cdot|Q_0'|-l_{Q'}\cdot|Q_0'|-2r_{Q'}\cdot l_{Q'},$$
hence the weight function $w_\Theta(X^\prime)$  is given by
 \begin{equation*}
   \begin{split}
   \theta(Q')=\theta_{\textrm{ind}}(Q')+\theta_{\text{add}}(Q')&=-r_{Q'}\cdot|Q_0'|-l_{Q'}\cdot|Q_0'|-2r_{Q'}\cdot l_{Q'}<0.
   \end{split}
 \end{equation*}

If $Q'$ is not connected, we denote its connected components as $Q^1, Q^2, \cdots Q^s$ which correspond to the direct summand $X^i$ of the representation $X^\prime$, then
 $$w_\Theta(X')=\sum\limits_{i=1}^sw_\Theta(X^i).$$
For each summand we have shown it is negative.
\end{proof}

\subsection{Proof of the Main Theorem}
We complete our proof of the main theorem by the following lemma.

\begin{lemma}
Let $Q$ be a quiver of $A_n$-type, then every indecomposable representation is stable with respect to the intrinsic weight function and rank function.
\end{lemma}

\begin{proof}
Let $X$ be an indecomposable representation of $Q$ and $X'\subset X$ be any proper subrepresentation.
Now $X$ must be of type $I_{p,q}$ with support quiver $Q^X:=\text{supp}_X(Q)$  connected, and  the support quiver $Q^{X'}:=\text{supp}_{X'}(Q)$ of $X'$ is a proper full subquiver of $Q^X$. Therefore our aim is to prove the following inequality holds for any proper full subquiver $Q^{X'}$ of $Q^X$:
\begin{equation}\label{2.3}
  \frac{w_\Theta(X^\prime)}{|Q^{X'}_0|}<\frac{w_\Theta(X)}{|Q^X_0|}.
\end{equation}

Let $Q^{X'}$ has $s$ connected components $Q^1, \cdots, Q^s$, clearly each $Q^i$ is a proper full subquiver of $Q^X$. To calculate the total weights $w_\Theta(X)$ and $w_\Theta(X^\prime)$, similar as the proof of Lemma \ref{subweight}, we first separate $Q^X$ from the whole quiver $Q$ to get the sums $\theta_{\text{ind}}(Q^X)$ ($=0$) and $\theta_{\text{ind}}(Q^{X'})$ ($<0$) coming from the independent weight system on $Q^X$. Secondly, we calculate the sums $\theta_{\text{add}}(Q^X)$ and $\theta_{\text{add}}(Q^{X'})$ when the rest parts of $Q$ are considered.

According to the relation of $Q^i$ and $Q^{X'}$, we can divide our consideration  into three different big cases, and when we take  the rest part of $Q$ into account, each big case can be divided into four different small cases.

Case I: all $Q^i$ are in the interior of $Q^X$, illustrated as follows:
$$\boxed{\cdots\bullet\longrightarrow\boxed{Q^1}\longleftarrow\bullet\cdots\bullet\longrightarrow\boxed{Q^2}\longleftarrow\bullet\cdots\cdots\bullet\longrightarrow
\boxed{Q^s}\longleftarrow\bullet\cdots}$$

Case II: there is a full subquiver of $Q^X$ (without loss of generality, assumed to be  $Q^s$) that shares  one boundary vertex with $Q^X$, illustrated as follows:
$$\boxed{\cdots\bullet\longrightarrow\boxed{Q^1}\longleftarrow\bullet\cdots\bullet\longrightarrow\boxed{Q^2}\longleftarrow\bullet\cdots\cdots\bullet\longrightarrow
\boxed{Q^s}}$$

Case III: there are two full subquivers of $Q^X$  (assumed to be  $Q^1$ and $Q^s$) each of which  has one boundary vertex coincides with that of  $Q^X$, illustrated as follows:
$$\boxed{\boxed{Q^1}\longleftarrow\bullet\cdots\bullet\longrightarrow\boxed{Q^2}\longleftarrow\bullet\cdots\cdots\bullet\longrightarrow
\boxed{Q^s}}$$

$\bullet$ We first consider the  Case I, when we added the rest of $Q$, there are   following four different cases:

$(a)$ the two arrows near $Q^X$   both point into $Q^X$:
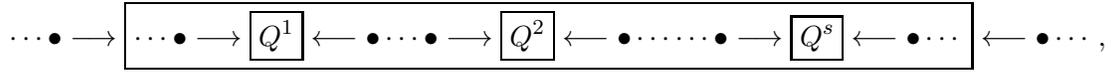
$$\cdots\bullet\longrightarrow\boxed{\cdots\bullet\longrightarrow\boxed{Q^1}\longleftarrow\bullet\cdots\bullet\longrightarrow\boxed{Q^2}\longleftarrow\bullet\cdots\cdots\bullet\longrightarrow
\boxed{Q^s}\longleftarrow\bullet\cdots}\longleftarrow\bullet\cdots,$$

$(b)$ the left arrow near $Q^X$ is  out from $Q^X$ and the right one points into $Q^X$:
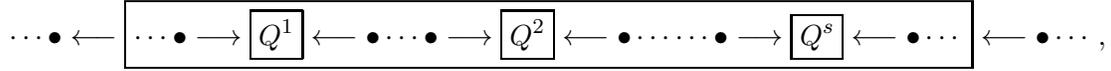
$$\cdots\bullet\longleftarrow\boxed{\cdots\bullet\longrightarrow\boxed{Q^1}\longleftarrow\bullet\cdots\bullet\longrightarrow\boxed{Q^2}\longleftarrow\bullet\cdots\cdots\bullet\longrightarrow
\boxed{Q^s}\longleftarrow\bullet\cdots}\longleftarrow\bullet\cdots,$$

$(c)$ the two arrows near $Q^X$ are  both out from $Q^X$:
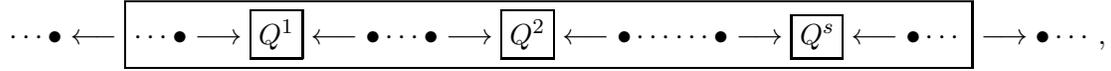
$$\cdots\bullet\longleftarrow\boxed{\cdots\bullet\longrightarrow\boxed{Q^1}\longleftarrow\bullet\cdots\bullet\longrightarrow\boxed{Q^2}\longleftarrow\bullet\cdots\cdots\bullet\longrightarrow
\boxed{Q^s}\longleftarrow\bullet\cdots}\longrightarrow\bullet\cdots,$$

$(d)$ the left arrow near $Q^X$ points into $Q^X$ and the right one is out from $Q^X$:
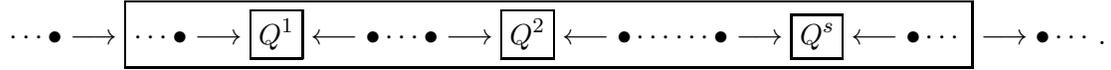
$$\cdots\bullet\longrightarrow\boxed{\cdots\bullet\longrightarrow\boxed{Q^1}\longleftarrow\bullet\cdots\bullet\longrightarrow\boxed{Q^2}\longleftarrow\bullet\cdots\cdots\bullet\longrightarrow
\boxed{Q^s}\longleftarrow\bullet\cdots}\longrightarrow\bullet\cdots.$$

No matter in what case, the slope of $X'$ is the same, to show $\mu(X')<\mu(X)$, we just need to consider the  case (a), since the value of  $\mu(X)$ is the minimum among these cases. Now we have

$$\mu(X')=\frac{w_\Theta(X^\prime)}{|Q^{X'}_0|}=\frac{\sum\limits_{i=1}^sw_\Theta(X^i)}{\sum\limits_{i=1}^s|Q^i_0|}=
\frac{\sum\limits_{i=1}^s\theta_{\text{ind}}(Q^i)+\sum\limits_{i=1}^s\theta_{\text{add}}(Q^i)}{\sum\limits_{i=1}^s|Q^i_0|},$$
$$\mu(X)=\frac{w_\Theta(X)}{|Q^X_0|}=\frac{\theta_{\text{ind}}(Q^X)+\theta_{\text{add}}(Q^X)}{|Q^X_0|},$$
where
 $$\theta_{\text{add}}(Q^X)=-|Q_0^X|(l_{Q^X}+r_{Q^X})-2l_{Q^X}\cdot r_{Q^X},$$
 \begin{equation*}
   \begin{split}
     \theta_{\text{add}}(Q^i)&=-|Q_0^i|(l_{Q^X}+r_{Q^X})-2l_{Q^X}\cdot r_{Q^i}-2l_{Q^i}\cdot r_{Q^X}+2l_{Q^X}\cdot r_{Q^X}\\
     &=-|Q_0^i|(l_{Q^X}+r_{Q^X})-2l_{Q^X}\cdot(r_{Q^i}-r_{Q^X})-2l_{Q^i}\cdot r_{Q^X},\quad 1\leq i\leq s,
   \end{split}
 \end{equation*}
 thus
 $$\sum\limits_{i=1}^s\theta_{\text{add}}(Q^i)=-(\sum\limits_{i=1}^s|Q_0^i|)(l_{Q^X}+r_{Q^X})-2l_{Q^X}\cdot\sum\limits_{i=1}^s(r_{Q^i}-r_{Q^X})
 -2\sum\limits_{i=1}^sl_{Q^i}\cdot r_{Q^X}.$$
Then the inequality (\ref{2.3}) reads
\begin{equation*}
\begin{split}
  &\frac{\sum\limits_{i=1}^s\theta_{\text{ind}}(Q^i)-(\sum\limits_{i=1}^s|Q_0^i|)(l_{Q^X}+r_{Q^X})-2l_{Q^X}\cdot\sum\limits_{i=1}^s(r_{Q^i}-r_{Q^X})
 -2\sum\limits_{i=1}^sl_{Q^i}\cdot r_{Q^X}}{\sum\limits_{i=1}^s|Q^i_0|}\\
  <&\frac{-|Q_0^X|(l_{Q^X}+r_{Q^X})-2l_{Q^X}\cdot r_{Q^X}}{|Q_0^X|},
  \end{split}
\end{equation*}
or reads
\begin{equation}\label{2.4}
  \frac{\sum\limits_{i=1}^s\theta_{\text{ind}}(Q^i)}{\sum\limits_{i=1}^s|Q^i_0|}
  <\frac{2l_{Q^X}\cdot\sum\limits_{i=1}^s(r_{Q^i}-r_{Q^X})}{\sum\limits_{i=1}^s|Q^i_0|}+2\sum\limits_{i=1}^sl_{Q^i}\cdot r_{Q^X}(\frac{1}{\sum\limits_{i=1}^s|Q^i_0|}-\frac{1}{|Q_0^X|}).
\end{equation}
Since $\sum\limits_{i=1}^s\theta_{\text{ind}}(Q^i)<0$, $l_{Q^X}\geq0$, $r_{Q^X}\geq0$, $r_{Q^i}>r_{Q^X}$ and $\sum\limits_{i=1}^s|Q^i_0|<|Q_0^X|$,
the equality (\ref{2.4}) holds.

$\bullet$ For  the Case II, as the analysis process in Case I, if, we also have the following four different cases when  the rest of $Q$ is added:

$(a)$
$\cdots\bullet\longrightarrow\boxed{\cdots\bullet\longrightarrow\boxed{Q^1}\longleftarrow\bullet\cdots\bullet\longrightarrow\boxed{Q^2}\longleftarrow\bullet\cdots\cdots\bullet\longrightarrow
\boxed{Q^s}}\longleftarrow\bullet\cdots$,

$(b)$ $\cdots\bullet\longleftarrow\boxed{\cdots\bullet\longrightarrow\boxed{Q^1}\longleftarrow\bullet\cdots\bullet\longrightarrow\boxed{Q^2}\longleftarrow\bullet\cdots\cdots\bullet\longrightarrow
\boxed{Q^s}}\longleftarrow\bullet\cdots$,

$(c)$ $\cdots\bullet\longleftarrow\boxed{\cdots\bullet\longrightarrow\boxed{Q^1}\longleftarrow\bullet\cdots\bullet\longrightarrow\boxed{Q^2}\longleftarrow\bullet\cdots\cdots\bullet\longrightarrow
\boxed{Q^s}}\longrightarrow\bullet\cdots$,

$(d)$ $\cdots\bullet\longrightarrow\boxed{\cdots\bullet\longrightarrow\boxed{Q^1}\longleftarrow\bullet\cdots\bullet\longrightarrow\boxed{Q^2}\longleftarrow\bullet\cdots\cdots\bullet\longrightarrow
\boxed{Q^s}}\longrightarrow\bullet\cdots$.

Note that in $(a)$ and $(b)$, $\mu(X')$ is the same, however, $\mu(X)$ is smaller in $(a)$, so we just need to show $\mu(X')<\mu(X)$ in $(a)$. In $(c)$ and $(d)$, $\mu(X')$ is the same, however, $\mu(X)$ is smaller in $(d)$, so we just need to show $\mu(X')<\mu(X)$ in $(d)$.

For the  case $(a)$, we have
\begin{equation*}
  \begin{split}
    \theta_{\text{add}}(Q^X)&=-|Q_0^X|(l_{Q^X}+r_{Q^X})-2l_{Q^X}\cdot r_{Q^X},\\
 \theta_{\text{add}}(Q^i)&=-|Q_0^i|(l_{Q^X}+r_{Q^X})-2l_{Q^X}\cdot(r_{Q^i}-r_{Q^X})-2l_{Q^i}\cdot r_{Q^X},\quad 1\leq i\leq s-1,\\
 \theta_{\text{add}}(Q^s)&=-|Q_0^s|(l_{Q^X}+r_{Q^X})-2l_{Q^s}\cdot r_{Q^X}-2l_{Q^X}\cdot r_{Q^s}+2l_{Q^X}\cdot r_{Q^X}\\
                         &=-|Q_0^s|(l_{Q^X}+r_{Q^X})-2l_{Q^s}\cdot r_{Q^X},
  \end{split}
\end{equation*}
thus the inequality (\ref{2.3}) reads
\begin{equation*}
\begin{split}
  &\frac{\sum\limits_{i=1}^s\theta_{\text{ind}}(Q^i)-(\sum\limits_{i=1}^s|Q_0^i|)(l_{Q^X}+r_{Q^X})-2l_{Q^X}\cdot\sum\limits_{i=1}^s(r_{Q^i}-r_{Q^X})
 -2\sum\limits_{i=1}^sl_{Q^i}\cdot r_{Q^X}}{\sum\limits_{i=1}^s|Q^i_0|}\\
  <&\frac{-|Q_0^X|(l_{Q^X}+r_{Q^X})-2l_{Q^X}\cdot r_{Q^X}}{|Q_0^X|},
  \end{split}
\end{equation*}
or reads
\begin{equation}\label{2.5}
  \frac{\sum\limits_{i=1}^s\theta_{\text{ind}}(Q^i)}{\sum\limits_{i=1}^s|Q^i_0|}
  <\frac{2l_{Q^X}\cdot\sum\limits_{i=1}^s(r_{Q^i}-r_{Q^X})+2\sum\limits_{i=1}^sl_{Q^i}\cdot r_{Q^X}}{\sum\limits_{i=1}^s|Q^i_0|}-\frac{2l_{Q^X}\cdot r_{Q^X}}{|Q_0^X|}.
\end{equation}
This inequality is the same as (\ref{2.4}), thus holds true.

For the case $(d)$, we have
\begin{equation*}
  \begin{split}
    \theta_{\text{add}}(Q^X)&=|Q_0^X|(r_{Q^X}-l_{Q^X}),\\
 \theta_{\text{add}}(Q^i)&=-|Q_0^i|(l_{Q^X}+r_{Q^X})-2l_{Q^X}\cdot(r_{Q^i}-r_{Q^X})-2l_{Q^i}\cdot r_{Q^X},\quad 1\leq i\leq s-1,\\
 \theta_{\text{add}}(Q^s)&=-|Q_0^s|(l_{Q^X}-r_{Q^X})\\
                         &=-|Q_0^s|(l_{Q^X}+r_{Q^X})+2|Q_0^s|\cdot r_{Q^X},
  \end{split}
\end{equation*}
thus the inequality (\ref{2.3}) reads
\begin{equation*}
\begin{split}
  &\frac{\sum\limits_{i=1}^s\theta_{\text{ind}}(Q^i)-(\sum\limits_{i=1}^s|Q_0^i|)(l_{Q^X}+r_{Q^X})-2l_{Q^X}\cdot\sum\limits_{i=1}^s(r_{Q^i}-r_{Q^X})
 -2\sum\limits_{i=1}^sl_{Q^i}\cdot r_{Q^X}+2|Q_0^s|\cdot r_{Q^X}}{\sum\limits_{i=1}^s|Q^i_0|}\\
  <&\frac{|Q_0^X|(r_{Q^X}-l_{Q^X})}{|Q_0^X|},
  \end{split}
\end{equation*}
or reads
\begin{align}\label{2.6}
  \frac{\sum\limits_{i=1}^s\theta_{\text{ind}}(Q^i)}{\sum\limits_{i=1}^s|Q^i_0|}
  <2r_{Q^X}-\frac{2|Q_0^s|\cdot r_{Q^X}}{\sum\limits_{i=1}^s|Q^i_0|}+\frac{2l_{Q^X}\cdot\sum\limits_{i=1}^{s-1}(r_{Q^i}-r_{Q^X})+2\sum\limits_{i=1}^sl_{Q^i}\cdot r_{Q^X}}{\sum\limits_{i=1}^s|Q^i_0|},
\end{align}
which is true due to again $\sum\limits_{i=1}^s\theta_{\text{ind}}(Q^i)<0$, $l_{Q^X}\geq0$, $r_{Q^X}\geq0$, $r_{Q^i}>r_{Q^X}$ and $\sum\limits_{i=1}^s|Q^i_0|<|Q_0^X|$.

$\bullet$ Last we consider the Case III. The subcases (a) and (d) are similar to the cases I-(a) and II-(d) respectively.

For the case $(b)$
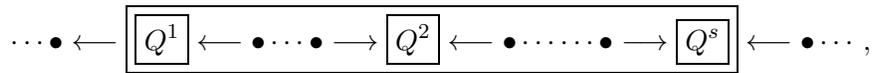
$$\cdots\bullet\longleftarrow\boxed{\boxed{Q^1}\longleftarrow\bullet\cdots\bullet\longrightarrow\boxed{Q^2}\longleftarrow\bullet\cdots\cdots\bullet\longrightarrow
\boxed{Q^s}}\longleftarrow\bullet\cdots,$$
we have
\begin{equation*}
  \begin{split}
    \theta_{\text{add}}(Q^X)&=|Q_0^X|(l_{Q^X}-r_{Q^X}),\\
    \theta_{\text{add}}(Q^1)&=|Q_0^1|(l_{Q^X}-r_{Q^X})=-|Q_0^1|(l_{Q^X}+r_{Q^X})+2|Q_0^1|\cdot l_{Q^X},\\
 \theta_{\text{add}}(Q^i)&=-|Q_0^i|(l_{Q^X}+r_{Q^X})-2l_{Q^X}\cdot(r_{Q^i}-r_{Q^X})-2l_{Q^i}\cdot r_{Q^X},\quad 2\leq i\leq s-1,\\
 \theta_{\text{add}}(Q^s)&=-|Q_0^s|(l_{Q^X}+r_{Q^X})-2l_{Q^s}\cdot r_{Q^X}-2l_{Q^X}\cdot r_{Q^s}+2l_{Q^X}\cdot r_{Q^X}\\
                         &=-|Q_0^s|(l_{Q^X}+r_{Q^X})-2l_{Q^s}\cdot r_{Q^X},
  \end{split}
\end{equation*}
thus the inequality (\ref{2.3}) reads
\begin{equation*}
\begin{split}
  &\frac{\sum\limits_{i=1}^s\theta_{\text{ind}}(Q^i)-(\sum\limits_{i=1}^s|Q_0^i|)(l_{Q^X}+r_{Q^X})-2l_{Q^X}\cdot\sum\limits_{i=2}^s(r_{Q^i}-r_{Q^X})
 -2\sum\limits_{i=2}^sl_{Q^i}\cdot r_{Q^X}+2|Q_0^1|\cdot l_{Q^X}}{\sum\limits_{i=1}^s|Q^i_0|}\\
  <&\frac{|Q_0^X|(l_{Q^X}-r_{Q^X})}{|Q_0^X|},
  \end{split}
\end{equation*}
or reads
\begin{equation}\label{2.7}
  \frac{\sum\limits_{i=1}^s\theta_{\text{ind}}(Q^i)}{\sum\limits_{i=1}^s|Q^i_0|}
  <2l_{Q^X}-\frac{2|Q_0^1|}{\sum\limits_{i=1}^s|Q^i_0|}\cdot l_{Q^X}+\frac{2l_{Q^X}\cdot\sum\limits_{i=2}^s(r_{Q^i}-r_{Q^X})+2\sum\limits_{i=2}^sl_{Q^i}\cdot r_{Q^X}}{\sum\limits_{i=1}^s|Q^i_0|},
\end{equation}
which  holds.

For the case $(c)$
$$\cdots\bullet\longleftarrow\boxed{\boxed{Q^1}\longleftarrow\bullet\cdots\bullet\longrightarrow\boxed{Q^2}\longleftarrow\bullet\cdots\cdots\bullet\longrightarrow
\boxed{Q^s}}\longrightarrow\bullet\cdots,$$
we have
\begin{equation*}
  \begin{split}
    \theta_{\text{add}}(Q^X)&=|Q_0^X|(l_{Q^X}+r_{Q^X})+2l_{Q^X}\cdot r_{Q^X},\\
    \theta_{\text{add}}(Q^1)&=|Q_0^1|(l_{Q^X}-r_{Q^X})=-|Q_0^1|(l_{Q^X}+r_{Q^X})+2|Q_0^1|\cdot l_{Q^X},\\
 \theta_{\text{add}}(Q^i)&=-|Q_0^i|(l_{Q^X}+r_{Q^X})-2l_{Q^X}\cdot(r_{Q^i}-r_{Q^X})-2l_{Q^i}\cdot r_{Q^X},\quad 2\leq i\leq s-1,\\
 \theta_{\text{add}}(Q^s)&=-|Q_0^s|(l_{Q^X}-r_{Q^X})=-|Q_0^s|(l_{Q^X}+r_{Q^X})-2|Q_0^s|\cdot r_{Q^X},
  \end{split}
\end{equation*}
thus the inequality (\ref{2.3}) reads
\begin{equation*}
\begin{split}
  &\frac{\sum\limits_{i=1}^s\theta_{\text{ind}}(Q^i)-(\sum\limits_{i=1}^s|Q_0^i|)(l_{Q^X}+r_{Q^X})-2l_{Q^X}\cdot\sum\limits_{i=2}^{s-1}(r_{Q^i}-r_{Q^X})
 -2\sum\limits_{i=2}^{s-1}l_{Q^i}\cdot r_{Q^X}+2|Q_0^1|\cdot r_{Q^X}-2|Q_0^s|\cdot r_{Q^X}}{\sum\limits_{i=1}^s|Q^i_0|}\\
  <&\frac{|Q_0^X|(l_{Q^X}+r_{Q^X})+2l_{Q^X}\cdot r_{Q^X}}{|Q_0^X|},
  \end{split}
\end{equation*}
or reads
\begin{equation}\label{2.8}
\begin{split}
  \frac{\sum\limits_{i=1}^s\theta_{\text{ind}}(Q^i)}{\sum\limits_{i=1}^s|Q^i_0|}
  <&2(l_{Q^X}+r_{Q^X})-\frac{2|Q_0^1|}{\sum\limits_{i=1}^s|Q^i_0|}\cdot r_{Q^X}\\&+\frac{2l_{Q^X}\cdot\sum\limits_{i=2}^{s-1}(r_{Q^i}-r_{Q^X})+2\sum\limits_{i=2}^{s-1}l_{Q^i}\cdot r_{Q^X}+2|Q_0^s|\cdot r_{Q^X}}{\sum\limits_{i=1}^s|Q^i_0|}+\frac{2l_{Q^X}\cdot r_{Q^X}}{|Q_0^X|},
  \end{split}
\end{equation}
which is satisfied.

So far, we complete the proof.
\end{proof}

\begin{definition}Let $\mathrm{Ind}(Q)$ be the (finite) set of the isomorphism classes indecomposable representations over $k$ of a Dynkin quiver $Q$. For a nonempty subset $U\subseteq \mathrm{Ind}(Q)$, one introduces
 a subset $S_\mathbb{Z}(Q,U)$ of $\mathbb{Z}^n$ with $n=|Q_0|$ as   
 \begin{align*}
  S_\mathbb{Z}(Q,U)=\left\{
  \begin{aligned} \Theta=(\theta_1,\cdots,\theta_n)\in\mathbb{Z}^n:\  &\textrm{each element of  } U \textrm{ is    stable}\textrm{ with respect to the corresponding } \\  \ & \textrm{weight function } w_\Theta\textrm{ and rank function }r
  \end{aligned} \right\}.
  \end{align*} 
 Obviously, $S_\mathbb{Z}(Q,U)\subset S_\mathbb{Z}(Q,V)$ if $V\subset U$. $S_\mathbb{Z}(Q,U)$ is determined by  finitely many linear inequalities $f_1(\Theta)>0,\cdots f_m(\Theta)>0$, then we define
a subset $S_\mathbb{R}(Q,U)$ of $\mathbb{R}^n$ as
$$S_\mathbb{R}(Q,U)=\{\Theta=(\theta_1,\cdots,\theta_n)\in\mathbb{R}^n: f_1(\Theta)>0,\cdots f_m(\Theta)>0 \},$$
and define a convex polyhedral cone $C(Q,U)$ of $\mathbb{R}^n$ as the closure  $$C(Q,U)=\overline{S_\mathbb{R}(Q,U)}.$$
The faces of maximal dimension in a cone  $C(Q,U)$ are called the  \emph{walls} in $\mathbb{R}^n$.
\end{definition}

\begin{corollary}(\cite{MR1}) Let $Q$ be q quiver of $A_n$-type, then the cardinality of $S_\mathbb{Z}(Q,\mathrm{Ind}(Q))$ is infinite. Moreover, for any element in  $S_\mathbb{Z}(Q,\mathrm{Ind}(Q))$, with respect to  the corresponding  weight  function and rank function
\begin{enumerate}
  \item any semistable representation is polystable;
  \item the Hader-Narasimhan strata  are precisely $GL(Q,d)$-orbits in $\mathrm{\mathbf{Rep}}(Q,d)$.
\end{enumerate}

\end{corollary}

\subsection{Revisit Intrinsic Weight System via Semi-Invariant Theory}
\begin{proposition}For each indecomposable representation $I_{p,q}$ of the quiver $Q$ of $A_n$-type, one defines  weight systems $\Theta(I_{p,q})$ and $\Theta^\prime(I_{p,q})$ as follows
\begin{align*}
  \Theta(I_{p,q})_i&=\left\{
                      \begin{array}{ll}
                        1, & \hbox{$p< i< q$ \textrm{ and } $i$\textrm{ is a vertex of type I},} \\
                        -1, & \hbox{$p<i<q$ \textrm{ and } $i$\textrm{ is a vertex of type II},}\\
0,& \hbox{$p<i<q$ \textrm{ and } $i$\textrm{ is a vertex of type III or IV},}\\
1, & \hbox{$i=p$ \textrm{ and } $i$\textrm{ is a vertex of type I or III}; $i=p=q$,}\\
0,& \hbox{$i=p<q$ \textrm{ and } $i$\textrm{ is a vertex of type II or IV},}\\
0, & \hbox{$i=p-1$ \textrm{ and } $i$\textrm{ is a vertex of type I or III},}\\
-1, & \hbox{$i=p-1$ \textrm{ and } $i$\textrm{ is a vertex of type II or IV},}\\
0, & \hbox{$i=q+1$ \textrm{ and } $i$\textrm{ is a vertex of type I or IV},}\\
-1, & \hbox{$i=q+1$ \textrm{ and } $i$\textrm{ is a vertex of type II or III},}\\
0, & \hbox{$i<p-1$; $i>q+1$},
                      \end{array}
                    \right.\\
\Theta^\prime(I_{p,q})_i&=\left\{
                      \begin{array}{ll}
                        1, & \hbox{$p< i< q$ \textrm{ and } $i$\textrm{ is a vertex of type I},} \\
                        -1, & \hbox{$p<i<q$ \textrm{ and } $i$\textrm{ is a vertex of type II},}\\
0,& \hbox{$p<i<q$ \textrm{ and } $i$\textrm{ is a vertex of type III or IV},}\\
-1, & \hbox{$i=p$ \textrm{ and } $i$\textrm{ is a vertex of type II or IV}; $i=p=q$,}\\
0,& \hbox{$i=p<q$ \textrm{ and } $i$\textrm{ is a vertex of type Ior III},}\\
1, & \hbox{$i=p-1$ \textrm{ and } $i$\textrm{ is a vertex of type I or III},}\\
0, & \hbox{$i=p-1$ \textrm{ and } $i$\textrm{ is a vertex of type II or IV},}\\
1, & \hbox{$i=q+1$ \textrm{ and } $i$\textrm{ is a vertex of type I or IV},}\\
0, & \hbox{$i=q+1$ \textrm{ and } $i$\textrm{ is a vertex of type II or III},}\\
0, & \hbox{$i<p-1$; $i>q+1$},
                      \end{array}
                    \right.
\end{align*}
then the intrinsic weight system $\Theta$ can be written as  $$\Theta=\sum_{I_{p,q}}c(I_{p,q})\Theta(I_{p,q})\ \ (\textrm{or \ } \Theta=\sum_{I_{p,q}}c(I_{p,q})\Theta^\prime(I_{p,q})),$$
where the sum runs through all indecomposable representations of $Q$, and the coefficients $c(I_{p,q})$'s are \textbf{non-negative integers}, moreover the sum can be taken over the indecomposable representations $I_{p,q}$ satisfying if $p\neq1$ is a vertex of type I or IV then $q=n$ or  $q\neq n$ is  of type II or IV; if  $p$ is a vertex of type II or III then  $q\neq n$ is   of type I or III; if  $p=1$ then  $q\neq n$ is a vertex of type I or III (or satisfying if $p\neq 1$ is a vertex of type II or III then $q=n$ or $q\neq n$ is one of type I or III;  if $p$ is a vertex of type I or IV then $q\neq n$ is one of type II or IV; if $p=1$ then $q\neq n$ is one of type II or IV).
\end{proposition}
\begin{proof}We prove this proposition by virtue of the semi-invariant theory. Let us first recall it briefly.
For a representation $X$ of a general quiver $Q$ with the dimension vector $d$ , every weight system $W=(W_i)\in \mathbb{Z}^{|Q_0|}$ on $Q$ defines a character $\chi_\Theta$ of reductive algebraic group $GL(Q,d)=\prod_{i\in Q_0}GL(d_i)$ acting on $X$  as a homomorphism
$$\chi_W:GL(Q,d)\rightarrow k^{\times}, g=(g_i:g_i\in GL(d_i))\mapsto \prod_{i\in Q_0}\det(g_i)^{W_i},$$
conversely, every character of $GL(d)$ must look like the above form. Let $$\mathrm{\mathbf{Rep}}(Q,d)=\bigoplus_{a\in Q_1}\Hom(k^{d_{s(a)}},k^{d_{t(a)}})$$ be the affine variety  of
representations of $Q$ with dimension vector $d$, a polynomial function $f$ in $k[\mathrm{\mathbf{Rep}}(Q,d)]$ is called a $W$-semi-invariant if $g\cdot f=\chi_W(g)f$ for any $g\in GL(X)$. Denote by $\mathrm{SI}_W(Q,d)$ the vector space  of $W$-semi-invariants, then the direct sum $$\mathrm{SI}(Q,d)=\bigoplus_{W\in \mathbb{Z}^{|Q_0|}}\mathrm{SI}_W(Q,d)$$
carries a ring structure, hence called the ring of semi-invariant, moreover $\mathrm{SI}(Q,d)=k[\mathrm{\mathbf{Rep}}(Q,d)]^{SL(Q,d)}$
for $SL(Q,d)=\prod_{i\in Q_0}SL(d_i)$ is the ring of polynomials in $k[\mathrm{\mathbf{Rep}}(Q,d)]$ which is stable under the action of $SL(d)$. Let $X$, $Y$ be two representations of a quiver $Q$ of $A_n$-type with dimension vectors $d_X, d_Y$ respectively, the Euler inner product
is given by
\begin{align*}
  \langle d_X, d_Y\rangle&=\dim_k\Hom_Q(X,Y)-\dim_k\mathrm{Ext}_Q(X,Y)\\
  &=\sum_{i\in Q_0}(d_X)_i(d_Y)_i-\sum_{a\in Q_1}(d_X)_{s(a)}(d_Y)_{t(a)}\\
  &=\sum_{i,i+1\in Q_0}((d_X)_i(d_Y)_i-(\widehat {d_X})_{ i}(\widehat{d_Y})_{ {i+1}}),
\end{align*}
where $i+1$ stands for the next vertex of $i$ along the reference direction, and
\begin{align*}
 (\widehat {d_X})_{ i}&=\left\{
                            \begin{array}{ll}
                              (d_X)_i, & \hbox{ $i$\textrm{ is a vertex of type I or  III},} \\
                              (d_X)_{i+1}, & \hbox{ $i$\textrm{ is a vertex of type II or  IV};}
                            \end{array}
                          \right.\\
(\widehat {d_Y})_{ i}&=\left\{
                            \begin{array}{ll}
                              (d_Y)_i, & \hbox{ $i$\textrm{ is a vertex of type II or  III},} \\
                              (d_Y)_{i-1}, & \hbox{ $i$\textrm{ is a vertex of type I or  IV}.}
                            \end{array}
                          \right.
\end{align*}
Define a map $f_X^Y:\bigoplus_{i\in Q_0}\Hom(X_i,Y_i)\rightarrow\bigoplus_{a\in Q_1}\Hom(X_{s(a)},Y_{t(a)})$ by
$$(f_i)_{i\in Q_0}\mapsto (f_{t(a)}X_a-Y_af_{s(a)})_{a\in Q_1}.$$
If $ \langle d_X, d_Y\rangle=0$, the matrix of $f^X_Y$ is a square matrix, then one can define a semi-invariant $c(X,Y)=\det f_X^Y$ of the action $GL(Q,d_X)\times GL(Q,d_Y)$ on $\mathrm{\mathbf{Rep}}(Q,d_X)\times \mathrm{\mathbf{Rep}}(Q,d_Y)$. For a fixed representation $X$ (or $Y$), the restriction $c(X,Y)$ to $\{X\}\times \mathrm{\mathbf{Rep}}(Q,d_Y)$ (or $\mathrm{\mathbf{Rep}}(Q,d_X)\times \{Y\}$) defines a semi-invariant $c_X(Y)$ (or $c^Y(X)$) in  $\mathrm{SI}(Q,d_Y)$  with respect to the weight system
$W_X=\{(W_X)_i\}_{i\in Q_0}$, where
 $$(W_X)_i=\langle d_X, d_i\rangle=\left\{
             \begin{array}{ll}
               (d_X)_i, & \hbox{$i$ \textrm{ is a vertex of type I};} \\
               (d_X)_i-(d_X)_{i-1}-(d_X)_{i+1}, & \hbox{$i$ \textrm{ is a vertex of type II};} \\
(d_X)_i-(d_X)_{i-1}, & \hbox{$i$ \textrm{ is a vertex of type III};} \\
(d_X)_i-(d_X)_{i+1}, & \hbox{$i$ \textrm{ is a vertex of type IV}}, \\
             \end{array}
           \right.
$$
for $(d_i)_j=\delta_{ij}$ (or $\mathrm{SI}(Q,d_X)$ with respect to the weight system
$W^Y=\{(W^Y)_i\}_{i\in Q_0}$, where
 $$(W^Y)_i=-\langle d_i,d_Y\rangle=\left\{
             \begin{array}{ll}
               -(d_Y)_i+(d_Y)_{i-1}+(d_X)_{i+1}, & \hbox{$i$ \textrm{ is a vertex of type I};} \\
               -(d_Y)_i, & \hbox{$i$ \textrm{ is a vertex of type II};} \\
-(d_Y)_i+(d_Y)_{i+1}, & \hbox{$i$ \textrm{ is a vertex of type III};} \\
-(d_Y)_i+(d_Y)_{i-1}, & \hbox{$i$ \textrm{ is a vertex of type IV}} \\
             \end{array}
           \right.).$$ Derksen and Weyman's remarkable theorem \cite{w} asserts that the semi-invariants of type $c_X(Y)$ (or $c^Y(X)$) span all the weight spaces in the  rings $\mathrm{SI}(Q,d_Y)$ (or $\mathrm{SI}(Q,d_X)$).
One can easily check
 $$W_{I_{p,q}}=\Theta(I_{p,q}),\ \ \ W^{I_{p,q}}=\Theta^\prime(I_{p,q}).$$
on the other hand,  $I_{1,n}$ is stable with respect to the intrinsic weight function, thus for each subrepresentations $R\subset I_{p,q}$ we have $w(R)<0$, then by King's result, there exists an $m>0$ and $f\in \mathrm{SI}(Q,\overrightarrow{n})_{m\Theta}$ such that $f(I_{1,n})\neq 0$, where $\overrightarrow{n}=(1,\cdots, 1)$ denotes for the dimension vector of $I_{1,n}$. Derksen-Weyman  theorem implies that the set $\Sigma(Q, d)=\{W:\mathrm{SI}(Q,d)_W\neq 0\}$ is saturated, therefore $\mathrm{SI}(Q,\overrightarrow{n})_{\Theta}\neq 0$. The ring $\mathrm{SI}(Q,\overrightarrow{n})$ is generated by all $c^{I_{p,q}}$'s with $\langle d_{I_{p,q}},\overrightarrow{n}\rangle=0$ (or $c_{I_{p,q}}$'s with $\langle\overrightarrow{n}, d_{I_{p,q}}\rangle=0$). All these facts together lead to the final conclusion.
\end{proof}

\begin{remark}We  write $\Theta=\Theta^+-\Theta^-$, where $\Theta^+=\{\Theta^+_i\}$
with $\Theta^+_i=\max\{\theta_i,0\}$ and $\Theta^-=\{\Theta^-_i\}$
with $\Theta^-_i=\max\{-\theta_i,0\}$. For a dimension vector $d$, if $\sum_{i\in Q_0}d_i\Theta_i\neq0$, then there is only trivial $\Theta$-semi-invariant. Therefore, we assume $\sum_{i\in Q_0}d_i\Theta_i=0$, i.e, $\sum_{i\in Q_0}d_i\Theta^+_i=\sum_{i\in Q_0}d_i\Theta^-_i=l$, then for a representation $X\in \mathrm{\mathbf{Rep}}(Q,d)$, one can define an $l\times l$ matrix
$$A:\bigoplus_{i\in Q_0}X_i^{\Theta^+_i}\rightarrow \bigoplus_{i\in Q_0}X_i^{\Theta^-_i},$$
where each block $A_{ij}\in\Hom(X_i,X_j)$ has a form
$$A_{ij}=\left\{
                  \begin{array}{ll}
                    X(p_{i,j}), & \hbox{if there exists a path $p_{i,j}$ from $i$ to $j$,} \\
                    0, & \hbox{\textrm{otherwise},}
                  \end{array}
                \right.
$$
with   $X(p_{i,j})$ denoting the composition of the morphisms $V_a$ for the arrows $a$'s consisting of the path $p_{i,j}$. Then $\det A$ is a semi-invariant in $\mathrm{SI}(Q,d)_\Theta$, and such semi-invariants generate the space $\mathrm{SI}(Q,d)_\Theta$.
\end{remark}

\noindent\textbf{Acknowledgements.}
 During the writing of this paper, we made several communications with Professors Markus Reineke,  Lutz Hille and Daniel Juteau. We would like to thank them for their useful suggestions.

\end{document}